\documentclass[12pt]{article} 
\usepackage{amsfonts,amsmath,latexsym,amssymb,mathrsfs,amsthm}
\usepackage{slashbox}
\usepackage{caption}

\let\OLDthebibliography\thebibliography
\renewcommand\thebibliography[1]{
  \OLDthebibliography{#1}
  \setlength{\parskip}{3pt}
  \setlength{\itemsep}{0pt plus 0.3ex}
}

\def\numberlikeadb{\global\def\theequation{\thesection.\arabic{equation}}}
\numberlikeadb
\newtheorem{theorem}{Theorem}[section]
\newtheorem{lemma}[theorem]{Lemma}

\newtheorem{remark}[theorem]{Remark}

\usepackage{lscape}
\usepackage{caption}
\usepackage{multirow}
\begin{document}

\title{Bounds for an integral of the modified Bessel function of the first kind and expressions involving it}
\author{Robert E. Gaunt\footnote{Department of Mathematics, The University of Manchester, Oxford Road, Manchester M13 9PL, UK}}  





\date{\today} 
\maketitle

\vspace{-5mm}

\begin{abstract}Simple upper and lower bounds are obtained for the integral $\int_0^x\mathrm{e}^{-\gamma t}t^\nu I_\nu(t)\,\mathrm{d}t$, $x>0$, $\nu>-\frac{1}{2}$, $0<\gamma<1$.  Most of our bounds for this integral are tight as $x\rightarrow\infty$.  We apply one of our inequalities to bound some expressions involving this integral. Two of these expressions appear in Stein's method for variance-gamma approximation, and our bounds will allow for a technical advancement to be made to the method. 
\end{abstract}


\noindent{{{\bf{AMS 2010 Subject Classification:}}} Primary 33C10; 26D15

\section{Introduction}\label{intro}

\subsection{Motivation through Stein's method for variance-gamma approximation}

Stein's method \cite{stein} is a powerful probabilistic technique for deriving bounds for distributional approximations with respect to a probability metric. It has found applications throughout the mathematical sciences in areas as diverse as queuing theory \cite{bj1},  number theory \cite{harper} and branching processes \cite{pekoz1}.  The method is particularly well developed for normal and Poisson approximation (see the books \cite{bhj92,chen, np12}), and there is active research into extensions to other distributional limits; see the survey \cite{ross}.

Recently, Stein's method has been extended to variance-gamma (VG) approximation \cite{eichelsbacher, gaunt vg, gaunt vg2}.  Applications have included VG approximation for a special case of the $D_2$ statistic from alignment-free sequence comparison \cite{bla, lippert}; quantitative six moment theorems for the VG approximation of double Wiener-It\^{o} integrals; and Laplace approximation of a random sum of independent mean zero random variables (see \cite{pike} for related results).  The VG distribution is commonly used in financial mathematics \cite{mcc98,madan}, and has a rich distributional theory, with special or limiting cases including the normal, gamma and Laplace distributions, and the product of two zero mean normals and difference of two gammas (see\cite{gaunt vg} and Chapter 4 of the book \cite{kkp01}, in which the distribution is called the generalized Laplace distribution). The VG distribution has also recently appeared in several other papers in the probability literature as a limiting distribution \cite{aaps17,ag18, azmooden, bt17}.   Extending Stein's method to the VG distribution is of interest because it puts some of the Stein's method literature into a more general framework and widens the scope of the method to treat new distributional limits.

Fundamental to Stein's method for VG approximation is the function $f_h:\mathbb{R}\rightarrow\mathbb{R}$ defined by
\begin{align}\label{vgsoln}f_h(x)&=-\frac{\mathrm{e}^{-\beta x} K_{\nu}(|x|)}{|x|^{\nu}} \int_0^x \mathrm{e}^{\beta t} |t|^{\nu} I_{\nu}(|t|) h(t) \,\mathrm{d}t  -\frac{\mathrm{e}^{-\beta x} I_{\nu}(|x|)}{|x|^{\nu}} \int_x^{\infty} \mathrm{e}^{\beta t} |t|^{\nu} K_{\nu}(|t|)h(t)\,\mathrm{d}t, 
\end{align}
where $x\in\mathbb{R}$, $\nu>-\frac{1}{2}$, $-1<\beta<1$, and $h:\mathbb{R}\rightarrow\mathbb{R}$ satisfies $\mu(h)=0$, for $\mu$ the VG probability measure. Here, $I_\nu(x)$ and $K_\nu(x)$ are modified Bessel functions of the first and second kind; basic properties of these functions that are needed in this paper are collected in Appendix \ref{appa}. In order to apply Stein's method for VG approximation, one must obtain uniform bounds, in terms of the supremum norms of $h$ and its derivatives,  for $f_h$ and certain lower order derivatives. New inequalities were obtained for the integrals given in (\ref{vgsoln}) by \cite{gaunt ineq1, gaunt ineq3} and applied in \cite{gaunt vg, dgv15} to obtain uniform bounds for derivatives of $f_h$ of arbitrary order, given sufficiently regular $h$. These bounds allow for VG approximations to be obtained by Stein's method in certain weak test function metrics, which imply convergence in distribution.

In order to obtain distributional approximations in the stronger and more widely used Kolmogorov and Wasserstein metrics, different types of bounds for $f_h$ and its derivatives are required than those given by \cite{gaunt vg, dgv15}. This was recently achieved by \cite{gaunt vg2} for a special case of the VG distribution, the symmetric VG distribution, that corresponds to setting $\beta=0$ in (\ref{vgsoln}).  The work of \cite{gaunt vg2} relied on new bounds of \cite{gaunt ineq3} for integrals of a similar form to those in (\ref{vgsoln}), as well as uniform bounds for some expressions involving these integrals. Uniform bounds were also obtained by \cite{gaunt ineq3} for a number of other expressions involving integrals of modified Bessel functions that correspond to the general $-1<\beta<1$ case for VG approximation. In particular, the following uniform bounds were established.  Suppose that $-1<\beta<0$ and $\nu\geq\frac{1}{2}$.  Then, for all $x\geq0$,
\begin{eqnarray}\label{term1}\frac{\mathrm{e}^{-\beta x}K_{\nu+1}(x)}{x^{\nu-1}}\int_0^x \mathrm{e}^{\beta t}t^{\nu}I_\nu(t)\,\mathrm{d}t&<& \frac{\nu+1}{(2\nu+1)(1+\beta)}, \\
\label{term2}\frac{\mathrm{e}^{-\beta x}K_{\nu}(x)}{x^{\nu-1}}\int_0^x \mathrm{e}^{\beta t}t^{\nu}I_\nu(t)\,\mathrm{d}t&<& \frac{\nu+1}{(2\nu+1)(1+\beta)}.
\end{eqnarray}
Uniform bounds for the case $0\leq\beta<1$ and $\nu>-\frac{1}{2}$ are easier to obtain and were also derived by \cite{gaunt ineq3}.  However, the case $-1<\beta<0$, $-\frac{1}{2}<\nu<\frac{1}{2}$ proved more challenging and \cite{gaunt ineq3} was unable to obtain uniform bounds in this parameter regime. This was left as an open problem, which, once solved, would allow for the uniform bounds of \cite{gaunt vg2} for $f_h$ to be extended from the $\beta=0$ case to the general $-1<\beta<1$ case. This would constitute a technical advancement that would mean that Stein's method for VG approximation could now be used to obtain Kolomogorov and Wasserstein distance bounds for the whole class of VG distributions. 

In this paper, we are able to solve the open problem and establish the desired uniform bounds for (\ref{term1}) and (\ref{term2}) in the remaining parameter regime of  $-1<\beta<0$, $-\frac{1}{2}<\nu<\frac{1}{2}$. Our results will be used in the forthcoming paper \cite{gaunt vg3} that will make the aforementioned technical advances to Stein's method for VG approximation that will allow explicit error bounds for VG approximations to be derived in the Kolmogorov and Wasserstein metrics.

\subsection{Summary of the paper}

Our approach to bounding the expressions (\ref{term1}) and (\ref{term2}) is to first obtain suitable bounds for the integral present in these terms, 
\begin{equation}\label{besint0}\int_0^x\mathrm{e}^{-\gamma t}t^\nu I_\nu(t)\,\mathrm{d}t, \quad x>0,\;\nu>-\tfrac{1}{2}, \;0<\gamma<1.
\end{equation}
Here $\gamma=-\beta$. A closed-form formula in terms of the modified Bessel function $I_\nu(x)$ and the modified Struve function $\mathbf{L}_{\nu}(x)$ does in fact exist in the case $\gamma=0$ \cite[formula 10.43.2]{olver}
\begin{equation}\label{besint6}\int x^{\nu}I_{\nu}(x)\,\mathrm{d}x =\sqrt{\pi}2^{\nu-1}\Gamma(\nu+\tfrac{1}{2})x\big(I_{\nu}(x)\mathbf{L}_{\nu-1}(x)-I_{\nu-1}(x)\mathbf{L}_{\nu}(x)\big), 
\end{equation}
and another closed-form formula is available for the case $\gamma=1$; see formula (\ref{intfor}) in Appendix \ref{appa}. However, there are no closed-form formulas involving modified Bessel and modified Struve functions for $0<\gamma<1$.  Moreover, the right-hand side of (\ref{besint6}) takes a fairly complicated form that is not suitable for bounding the expressions (\ref{term1}) and (\ref{term2}). This provides our motivation for establishing simple bounds, in terms of the modified Bessel function $I_\nu(x)$, for the integral (\ref{besint0}).

Some simple bounds for the integral (\ref{besint0}), involving the modified Bessel function of the first kind, have been been obtained in the recent papers \cite{gaunt ineq1,gaunt ineq3}. In Section \ref{sec2}, we establish a new upper bound for (\ref{besint0}) that holds in the restricted region $x\geq x_*$, for $x_*>\frac{1}{1-\gamma}$, but is crucially of the correct asymptotic order as $x\rightarrow\infty$, and is valid for all $\nu>-\frac{1}{2}$, $0<\gamma<1$ (the bounds of \cite{gaunt ineq1,gaunt ineq3} are only valid for $\nu\geq\frac{1}{2}$, $0<\gamma<1$). These features of the bound are precisely what we need in order to bound the expressions (\ref{term1}) and (\ref{term2}).
 We also obtain two other upper bounds for (\ref{besint0}), inequalities (\ref{ineqb10}) and (\ref{ineqb11}), which are the first upper bounds in the literature that are valid for all $x>0$, $\nu>-\frac{1}{2}$, $0<\gamma<1$.  We shall complement our upper bounds with several lower bounds for the integral. All of our lower bounds are tight as $x\rightarrow\infty$, and one of our lower bounds improves on the only other lower bound in the literature for (\ref{besint0}) (due to \cite{gaunt ineq3}). Our upper bound (\ref{ineqb1}) will have an immediate application to Stein's method for VG approximation. Due to the combination of the simple form and accuracy of our bounds, they may also prove useful in other problems involving modified Bessel functions; see, for example, \cite{bs09,baricz3} which uses inequalities for the modified Bessel function $I_\nu(x)$ to derive tight bounds for the generalized Marcum Q-function, which arises in radar signal processing.

In Section \ref{sec3}, we apply our upper bound (\ref{ineqb1}) for (\ref{besint0}), together with known inequalities for products of modified Bessel functions, to obtain uniform upper bounds for the expressions (\ref{term1}) and (\ref{term2}) in the parameter regime $\nu>-\frac{1}{2}$, $-1<\beta<0$. We also obtain a uniform upper bound for a related expression, which allows us, as a consequence, to prove our upper bound (\ref{ineqb10}) for the integral (\ref{besint0}), which is valid for all $x>0$, $\nu>-\frac{1}{2}$, $0<\gamma<1$.  We complement our upper bounds with lower bounds for the supremum over all $x\geq0$ for these expressions, which give useful insight into the accuracy of our upper bounds.

\section{Bounds for the integral}\label{sec2}

In the following Theorems \ref{tiger1} and \ref{tiger2}, we obtain new inequalities for the integral (\ref{besint0}).  These inequalities complement the inequalities of Theorem 2.1 of \cite{gaunt ineq1} and Theorem 2.3 of \cite{gaunt ineq3} for this integral, together with the inequalities of \cite{gaunt ineq6} for the related integral $\int_0^x \mathrm{e}^{-\gamma t}t^{-\nu}I_\nu(t)\,\mathrm{d}t$.

\begin{theorem}\label{tiger1}Let $0<\gamma<1$. Fix $x_*>\frac{1}{1-\gamma}$.  Then, for $x\geq x_*$,
\begin{equation}\label{ineqb1}\int_0^x \mathrm{e}^{-\gamma t}t^{\nu}I_\nu(t)\,\mathrm{d}t<M_{\nu,\gamma}(x_*)\mathrm{e}^{-\gamma x} x^\nu I_{\nu+1}(x), \quad \nu>-\tfrac{1}{2},
\end{equation}
where
\begin{equation}\label{mng}M_{\nu,\gamma}(x_*)=\max\bigg\{\frac{2(\nu+1+x_*)}{2\nu+1},\frac{x_*}{(1-\gamma)x_*-1}\bigg\}.
\end{equation}
Also, for $x>0$,
\begin{align}\label{ineqb2}\int_0^x \mathrm{e}^{-\gamma t}t^{\nu}I_\nu(t)\,\mathrm{d}t&>\frac{1}{1-\gamma}\bigg\{\mathrm{e}^{-\gamma x}x^\nu\bigg(I_\nu(x)-\frac{x^\nu}{\Gamma(\nu+1)2^\nu}\bigg)-\frac{\gamma(2\nu+1,\gamma x)}{\Gamma(\nu+1)2^\nu\gamma^{2\nu}}\bigg\}, \\
&\quad\quad\quad\quad\quad\quad\quad\quad\quad\quad\quad\quad\quad\quad\quad\quad\quad\quad\quad-\tfrac{1}{2}<\nu\leq0,\nonumber \\
\label{ineqb3}\int_0^x \mathrm{e}^{-\gamma t}t^{\nu}I_\nu(t)\,\mathrm{d}t&>\frac{1}{1-\gamma}\bigg(1-\frac{4\nu^2}{(2\nu-1)(1-\gamma)}\frac{1}{x}\bigg)\mathrm{e}^{-\gamma x}x^\nu I_\nu(x), \quad \nu\geq\tfrac{3}{2}, \\
\label{ineqb4}\int_0^x \mathrm{e}^{-\gamma t}t^{\nu}I_\nu(t)\,\mathrm{d}t&>\mathrm{e}^{-\gamma x}x^\nu\sum_{k=0}^\infty \gamma^k I_{\nu+k+1}(x), \quad \nu>-\tfrac{1}{2}, \\
\label{ineqb5}\int_0^x \mathrm{e}^{-\gamma t}I_0(t)\,\mathrm{d}t&>\frac{1}{1-\gamma}\mathrm{e}^{-\gamma x} (I_0(x)-1).
\end{align}
 Inequalities (\ref{ineqb2})--(\ref{ineqb5}) are tight as $x\rightarrow\infty$. In inequality (\ref{ineqb2}), $\gamma(a,x)=\int_0^xt^{a-1}\mathrm{e}^{-t}\,\mathrm{d}t$ is the lower incomplete gamma function.
\end{theorem}

The following inequalities can be proved as a consequence of some of the upper bounds in Theorem \ref{notfin} given in Section \ref{sec3}.  We therefore defer the proof of Theorem \ref{tiger2} until Section \ref{sec3}.

\begin{theorem}\label{tiger2}Let $0<\gamma<1$.  Then, for $x>0$,
\begin{align}\label{ineqb10}\int_0^x\mathrm{e}^{-\gamma t}t^\nu I_\nu(t)\,\mathrm{d}t&<\frac{2(2\nu+7)}{(2\nu+1)(1-\gamma)}\mathrm{e}^{-\gamma x}x^\nu I_{\nu+1}(x), \quad \nu>-\tfrac{1}{2}, \\
\label{ineqb11}\int_0^x\mathrm{e}^{-\gamma t}t^\nu I_\nu(t)\,\mathrm{d}t&<\frac{2\nu+7}{(2\nu+1)(1-\gamma)}\mathrm{e}^{-\gamma x}x^\nu I_\nu(x), \quad \nu>-\tfrac{1}{2},\\
\label{ineqb12}\int_0^x\mathrm{e}^{-\gamma t}t^\nu I_\nu(t)\,\mathrm{d}t&>\frac{1}{1-\gamma}\bigg\{1-\frac{4\nu(2\nu+5)}{(2\nu-1)(1-\gamma)}\frac{1}{x}\bigg\}\mathrm{e}^{-\gamma x}x^\nu I_\nu(x), \quad \nu>\tfrac{1}{2}.
\end{align}
Inequality (\ref{ineqb12}) is tight as $x\rightarrow\infty$. 
\end{theorem}

\noindent{\emph{Proof of Theorem \ref{tiger1}.}} (i) Fix $x_*>\frac{1}{1-\gamma}$.
Consider the function
\begin{equation*}u(x)=M_{\nu,\gamma}(x_*)\mathrm{e}^{-\gamma x}x^\nu I_{\nu+1}(x)-\int_0^x \mathrm{e}^{-\gamma t}t^\nu I_\nu(t)\,\mathrm{d}t.
\end{equation*}
We shall argue that $u(x)>0$ for all $x\geq x_*$, which will prove inequality (\ref{ineqb1}).  

Let us first prove that $u(x_*)>0$. Consider now the function
\begin{equation*}v(x)=\frac{\mathrm{e}^{\gamma x}}{x^\nu I_{\nu+1}(x)}\int_0^x\mathrm{e}^{-\gamma t} t^\nu I_\nu(t)\,\mathrm{d}t.
\end{equation*}
We shall show that $v(x_*)< M_{\nu,\gamma}(x_*)$, which will prove that $u(x_*)>0$. We have that
\begin{equation*}\frac{\partial v(x)}{\partial \gamma}=\frac{\mathrm{e}^{\gamma x}}{x^\nu I_{\nu+1}(x)}\int_0^x(x-t)\mathrm{e}^{-\gamma t} t^\nu I_\nu(t)\,\mathrm{d}t>0,
\end{equation*}
and therefore
\begin{align*}v(x)< \frac{\mathrm{e}^{ x}}{x^\nu I_{\nu+1}(x)}\int_0^x\mathrm{e}^{- t} t^\nu I_\nu(t)\,\mathrm{d}t=\frac{x}{2\nu+1}\bigg(\frac{I_\nu(x)}{I_{\nu+1}(x)}+1\bigg),
\end{align*}
where we evaluated the integral using (\ref{intfor}). We bound the ratio of modified Bessel functions of the first kind using the inequality 
\begin{equation*}\frac{I_{\mu+1}(x)}{I_{\mu}(x)}>\frac{x}{2(\mu+1)+x}, \quad x>0,\;\mu>-1,
\end{equation*}
which is the simplest lower bound in a sequence of rational bounds given in \cite{nasell2}. Applying this inequality gives us the desired bound
\begin{align*}v(x_*)<\frac{x_*}{2\nu+1}\bigg(\frac{2(\nu+1)+x_*}{x_*}+1\bigg)=\frac{2(\nu+1+x_*)}{2\nu+1}\leq M_{\nu,\gamma}(x_*).
\end{align*}

We now prove that $u'(x)>0$ for $x>x_*$, which will complete the proof.  A simple calculation using the differentiation formula (\ref{diffone}) gives that
\begin{align*}u'(x)&=M_{\nu,\gamma}(x_*)\frac{\mathrm{d}}{\mathrm{d}x}\big(\mathrm{e}^{-\gamma x}x^{-1}\cdot x^{\nu+1} I_{\nu+1}(x)\big)- \mathrm{e}^{-\gamma x}x^\nu I_\nu(x)\\
&=M_{\nu,\gamma}(x_*)\mathrm{e}^{-\gamma x}x^\nu\big(I_{\nu}(x)-x^{-1}I_{\nu+1}(x)-\gamma I_{\nu+1}(x)\big)-\mathrm{e}^{-\gamma x}x^\nu I_\nu(x).
\end{align*}
Using inequality (\ref{Imon}) now gives us the inequality
\begin{align*}u'(x)&>M_{\nu,\gamma}(x_*)\mathrm{e}^{-\gamma x}x^\nu\big(1-\gamma -x^{-1})I_\nu(x)-\mathrm{e}^{-\gamma x}x^\nu I_\nu(x)\\
&\geq\bigg(\frac{1-\gamma-x^{-1}}{1-\gamma-x_*^{-1}}-1\bigg)\mathrm{e}^{-\gamma x}x^\nu I_\nu(x)>0,
\end{align*}
for $x>x_*$, as required.

\vspace{2mm}

\noindent{(ii)} Now let $x>0$ and suppose $-\frac{1}{2}<\nu\leq0$.  Then, by integration by parts and the differentiation formula (\ref{diffone}), we have
\begin{align*}\int_0^x\mathrm{e}^{-\gamma t}\bigg(t^\nu I_\nu(t)-\frac{t^{2\nu}}{\Gamma(\nu+1)2^\nu}\bigg)\,\mathrm{d}t&=-\frac{1}{\gamma}\mathrm{e}^{-\gamma x}\bigg(x^\nu I_\nu(x)-\frac{x^{2\nu}}{\Gamma(\nu+1)2^\nu}\bigg)\\
&\quad+\frac{1}{\gamma}\int_0^x \mathrm{e}^{-\gamma t}\bigg(t^\nu I_{\nu-1}(t)-\frac{2\nu t^{2\nu-1}}{\Gamma(\nu+1) 2^\nu}\bigg)\,\mathrm{d}t,
\end{align*}
where the integrals can be seen to exist for $\nu>-\frac{1}{2}$ by (\ref{Itend0}) and the standard identity $\Gamma(x+1)=x\Gamma(x)$.  We also used (\ref{Itend0}) to compute the limit $\lim_{x\downarrow0}\big(x^\nu I_\nu(x)-\frac{x^{2\nu}}{\Gamma(\nu+1)2^\nu}\big)=0$. We can rearrange to get
\begin{align*}&\int_0^x \mathrm{e}^{-\gamma t}\bigg(t^\nu I_{\nu-1}(t)-\frac{2\nu t^{2\nu-1}}{\Gamma(\nu+1) 2^\nu}\bigg)\,\mathrm{d}t-\gamma\int_0^x\mathrm{e}^{-\gamma t}\bigg(t^\nu I_\nu(t)-\frac{t^{2\nu}}{\Gamma(\nu+1)2^\nu}\bigg)\,\mathrm{d}t\\
&\quad=\mathrm{e}^{-\gamma x}\bigg(x^\nu I_\nu(x)-\frac{x^{2\nu}}{\Gamma(\nu+1)2^\nu}\bigg),
\end{align*}
and then using the identity (\ref{Iidentity}) gives that
\begin{align}&\int_0^x \mathrm{e}^{-\gamma t}t^\nu I_{\nu+1}(t)\,\mathrm{d}t+2\nu\int_0^x\mathrm{e}^{-\gamma t}\bigg(t^{\nu-1}I_\nu(t)-\frac{t^{2\nu-1}}{\Gamma(\nu+1)2^\nu}\bigg)\,\mathrm{d}t-\gamma\int_0^x\mathrm{e}^{-\gamma t}t^\nu I_\nu(t)\,\mathrm{d}t\nonumber\\
\label{55555}&\quad=\mathrm{e}^{-\gamma x}\bigg(x^\nu I_\nu(x)-\frac{x^{2\nu}}{\Gamma(\nu+1)2^\nu}\bigg)-\gamma \int_0^x\mathrm{e}^{-\gamma t}\frac{t^{2\nu}}{\Gamma(\nu+1)2^\nu}\,\mathrm{d}t.
\end{align}
We now note that, by (\ref{Idefn}),
\begin{equation}\label{jjj}t^{\nu-1}I_\nu(t)-\frac{t^{2\nu-1}}{\Gamma(\nu+1)2^\nu}=\sum_{k=1}^{\infty} \frac{(\frac{1}{2}t)^{\nu+2k-1}}{\Gamma(\nu +k+1) k!}>0, \quad t>0.
\end{equation}
Therefore, by using inequality (\ref{jjj}) and that $-\frac{1}{2}<\nu\leq0$ to bound the second integral and inequality (\ref{Imon}) to bound the first integral, we obtain
\begin{align*}\int_0^x \mathrm{e}^{-\gamma t}t^\nu I_{\nu}(t)\,\mathrm{d}t>\frac{1}{1-\gamma}\bigg\{\mathrm{e}^{-\gamma x}\bigg(x^\nu I_\nu(x)-\frac{x^{2\nu}}{\Gamma(\nu+1)2^\nu}\bigg)-\gamma \int_0^x\mathrm{e}^{-\gamma t}\frac{t^{2\nu}}{\Gamma(\nu+1)2^\nu}\,\mathrm{d}t\bigg\},
\end{align*}
and we finally arrive at inequality (\ref{ineqb2}) by using a change of variable to evaluate the integral $\int_0^x\mathrm{e}^{-\gamma t}t^{2\nu}\,\mathrm{d}t=\frac{1}{\gamma^{2\nu+1}}\gamma(2\nu+1,\gamma x)$.

\vspace{2mm}

\noindent{(iii)}  An application of integration by parts gives that
\begin{align}\label{jj256}\int_0^x \mathrm{e}^{-\gamma t}t^{\nu}I_{\nu}(t)\,\mathrm{d}t&=-\frac{1}{\gamma}\mathrm{e}^{-\gamma x}x^{\nu}I_{\nu}(x)+\frac{1}{\gamma}\int_0^x\mathrm{e}^{-\gamma t}t^{\nu}I_{\nu-1}(t)\,\mathrm{d}t,
\end{align}
where we used that $\lim_{x\downarrow0}x^{\nu}I_{\nu}(x)=0$ for $\nu\geq\frac{3}{2}$ (see (\ref{Itend0})) and the differentiation formula (\ref{diffone}). Rearranging (\ref{jj256}) and using the identity (\ref{Iidentity}) gives that
\begin{align}\label{1stint}\int_0^x \mathrm{e}^{-\gamma t}t^{\nu}I_{\nu+1}(t)\,\mathrm{d}t-\gamma \int_0^x \mathrm{e}^{-\gamma t}t^{\nu}I_{\nu}(t)\,\mathrm{d}t=\mathrm{e}^{-\gamma x}x^\nu I_\nu(x)-2\nu\int_0^x \mathrm{e}^{-\gamma t}t^{\nu-1}I_{\nu}(t)\,\mathrm{d}t.
\end{align}
Using (\ref{Imon}) to bound the first integral in (\ref{1stint}), followed by a rearrangement and then another application of inequality (\ref{Imon}) gives the inequality
\begin{align}\int_0^x \mathrm{e}^{-\gamma t}t^{\nu}I_{\nu}(t)\,\mathrm{d}t&>\frac{1}{1-\gamma}\bigg\{\mathrm{e}^{-\gamma x}x^\nu I_\nu(x)-2\nu\int_0^x \mathrm{e}^{-\gamma t}t^{\nu-1}I_{\nu}(t)\,\mathrm{d}t\bigg\}\nonumber\\
\label{1stint0}&>\frac{1}{1-\gamma}\bigg\{\mathrm{e}^{-\gamma x}x^\nu I_\nu(x)-2\nu\int_0^x \mathrm{e}^{-\gamma t}t^{\nu-1}I_{\nu-1}(t)\,\mathrm{d}t\bigg\}.
\end{align}
We now recall an inequality that is immediate from inequality (2.19) of \cite{gaunt ineq3}: for $x>0$,
\begin{equation}\label{gaunt219}\int_0^x \mathrm{e}^{-\gamma t}t^{\mu}I_{\mu}(t)\,\mathrm{d}t<\frac{2(\mu+1)}{2\mu+1}\mathrm{e}^{-\gamma x}x^\mu I_{\mu+1}(x), \quad\mu\geq\tfrac{1}{2},\;0<\gamma<1.
\end{equation}
Applying this inequality to (\ref{1stint0}) yields inequality (\ref{ineqb3}), as required.

\vspace{2mm}

\noindent{(iv)} Let $\nu>-\frac{1}{2}$, which will ensure that all integrals that appear in this proof of inequality (\ref{ineqb4}) exist.  We begin with a similar integration by parts to part (iii):
\begin{align}\label{jj25}\int_0^x \mathrm{e}^{-\gamma t}t^{\nu+1}I_{\nu+1}(t)\,\mathrm{d}t&=-\frac{1}{\gamma}\mathrm{e}^{-\gamma x}x^{\nu+1}I_{\nu+1}(x)+\frac{1}{\gamma}\int_0^x\mathrm{e}^{-\gamma t}t^{\nu+1}I_\nu(t)\,\mathrm{d}t,
\end{align}
where we used that $\lim_{x\downarrow0}x^{\nu+1}I_{\nu+1}(x)=0$ for $\nu>-\frac{1}{2}$ (see (\ref{Itend0})) and the differentiation formula (\ref{diffone}). Rearranging (\ref{jj25}) gives
\begin{align}\label{jj26}\int_0^x\mathrm{e}^{-\gamma t}t^{\nu+1}I_\nu(t)\,\mathrm{d}t=\mathrm{e}^{-\gamma x}x^{\nu+1}I_{\nu+1}(x)+\gamma\int_0^x \mathrm{e}^{-\gamma t}t^{\nu+1}I_{\nu+1}(t)\,\mathrm{d}t.
\end{align}
We now note that, for $x>0$, $\int_0^x\mathrm{e}^{-\gamma t}t^{\nu+1}I_\nu(t)\,\mathrm{d}t<x\int_0^x\mathrm{e}^{-\gamma t}t^{\nu}I_\nu(t)\,\mathrm{d}t$, which holds because $I_\nu(x)>0$ for $x>0$, $\nu>-\frac{1}{2}$.  Applying this inequality to (\ref{jj26}) yields
\begin{align}\label{jj27}\int_0^x\mathrm{e}^{-\gamma t}t^{\nu}I_\nu(t)\,\mathrm{d}t>\mathrm{e}^{-\gamma x}x^{\nu}I_{\nu+1}(x)+\frac{\gamma}{x}\int_0^x \mathrm{e}^{-\gamma t}t^{\nu+1}I_{\nu+1}(t)\,\mathrm{d}t.
\end{align}
From (\ref{jj27}) we get another inequality
\begin{align*}\int_0^x\mathrm{e}^{-\gamma t}t^{\nu}I_\nu(t)\,\mathrm{d}t&>\mathrm{e}^{-\gamma x}x^{\nu}I_{\nu+1}(x)+\frac{\gamma}{x}\bigg(\mathrm{e}^{-\gamma x}x^{\nu+1}I_{\nu+2}(x)+\frac{\gamma}{x}\int_0^x \mathrm{e}^{-\gamma t}t^{\nu+2}I_{\nu+2}(t)\,\mathrm{d}t\bigg)\\
&=\mathrm{e}^{-\gamma x}x^{\nu}I_{\nu+1}(x)+\gamma \mathrm{e}^{-\gamma x}x^{\nu}I_{\nu+2}(x)+\frac{\gamma^2}{x^2}\int_0^x \mathrm{e}^{-\gamma t}t^{\nu+2}I_{\nu+2}(t)\,\mathrm{d}t.
\end{align*}
Iterating this procedure then yields inequality (\ref{ineqb4}).  In applying this iteration, it should be noted that the series $\sum_{k=0}^\infty \gamma^k I_{\nu+k+1}(x)$ is absolutely convergent.  To see this, we can repeatedly use inequality (\ref{Imon}) (as $\nu>-\frac{1}{2}$) to obtain that, for all $x>0$,
\[\sum_{k=0}^\infty \gamma^k I_{\nu+k+1}(x)<I_{\nu+1}(x)\sum_{k=0}^\infty \gamma^k=\frac{I_{\nu+1}(x)}{1-\gamma}, \]
and the geometric series converges because $0<\gamma<1$.

\vspace{2mm}

\noindent{(v)}  By integration by parts, we have
\begin{align}\int_0^x \mathrm{e}^{-\gamma t}I_0(t)\,\mathrm{d}t&=-\frac{1}{\gamma}\mathrm{e}^{-\gamma x}\big(I_0(x)-1\big)+\frac{1}{\gamma}\int_0^x\mathrm{e}^{-\gamma t}I_1(t)\,\mathrm{d}t \nonumber\\
\label{j5}&<-\frac{1}{\gamma}\mathrm{e}^{-\gamma x}\big(I_0(x)-1\big)+\frac{1}{\gamma}\int_0^x\mathrm{e}^{-\gamma t}I_0(t)\,\mathrm{d}t,
\end{align}
where in the first step we used that $I_0(0)=1$ (this is readily seen from (\ref{Idefn})) and the differentiation formula (\ref{diff0}), and we used inequality (\ref{Imon}) to obtain the inequality. Rearranging (\ref{j5}) yields inequality (\ref{ineqb5}).

\vspace{2mm}

\noindent{(vi)} Finally, we prove that inequalities (\ref{ineqb2})--(\ref{ineqb5}) are tight as $x\rightarrow\infty$.  We start by noting that a straightforward asymptotic analysis using the limiting form (\ref{Itendinfinity}) gives that, for $\nu>-\frac{1}{2}$ and $0<\gamma<1$, 
\begin{equation}\label{eqeq1} \int_0^x \mathrm{e}^{-\gamma t}t^\nu  I_{\nu}(t)\,\mathrm{d}t\sim \frac{1}{\sqrt{2\pi}(1-\gamma)}x^{\nu-1/2}\mathrm{e}^{(1-\gamma)x}, \quad x\rightarrow\infty,
\end{equation}
and we also have, for $n\in\mathbb{R}$,
\begin{equation}\label{eqeq2}\mathrm{e}^{-\gamma x}x^\nu I_{\nu+n}(x)\sim  \frac{1}{\sqrt{2\pi}}x^{\nu-1/2}\mathrm{e}^{(1-\gamma)x}, \quad x\rightarrow\infty.
\end{equation}
That inequalities (\ref{ineqb2}), (\ref{ineqb3}) and (\ref{ineqb5}) are tight as $x\rightarrow\infty$ follows directly from  (\ref{eqeq1}) and (\ref{eqeq2}). As an example, for inequality (\ref{ineqb5}), we have, as $x\rightarrow\infty$,
\begin{align*}\int_0^x \mathrm{e}^{-\gamma t}I_0(t)\,\mathrm{d}t\sim \frac{\mathrm{e}^{(1-\gamma)x}}{(1-\gamma)\sqrt{2\pi x}}\quad \text{and} \quad  \frac{1}{1-\gamma}\mathrm{e}^{-\gamma x} (I_0(x)-1)\sim \frac{\mathrm{e}^{(1-\gamma)x}}{(1-\gamma)\sqrt{2\pi x}},
\end{align*}
and the tightness of inequalities (\ref{ineqb2}) and (\ref{ineqb3}) is established similarly. For the tightness of inequality (\ref{ineqb4}) we just need to additionally note that $\sum_{k=0}^\infty\gamma^k=\frac{1}{1-\gamma}$, as $0<\gamma<1$. \hfill $\square$

\begin{remark}Let $0<\gamma<1$. Then the following inequalities hold. For $x>0$,
\begin{align}\int_0^x \mathrm{e}^{-\gamma t}t^{\nu}I_{\nu+1}(t)\,\mathrm{d}t&>\frac{1}{1-\gamma}\bigg\{\mathrm{e}^{-\gamma x}x^\nu\bigg(I_\nu(x)-\frac{x^\nu}{\Gamma(\nu+1)2^\nu}\bigg)-\frac{\gamma(2\nu+1,\gamma x)}{\Gamma(\nu+1)2^\nu\gamma^{2\nu}}\bigg\},\nonumber \\
\label{ineqb21}&\quad\quad\quad\quad\quad\quad\quad\quad\quad\quad\quad\quad\quad\quad\quad\quad\quad\quad\quad-\tfrac{1}{2}<\nu\leq0,\\
\label{ineqb22}\int_0^x \mathrm{e}^{-\gamma t}t^{\nu}I_{\nu+1}(t)\,\mathrm{d}t&>\frac{1}{1-\gamma}\bigg(1-\frac{4\nu^2}{(2\nu-1)(1-\gamma)}\frac{1}{x}\bigg)\mathrm{e}^{-\gamma x}x^\nu I_\nu(x), \quad \nu\geq\tfrac{3}{2}, \\
\label{ineqb23}\int_0^x\mathrm{e}^{-\gamma t}t^\nu I_{\nu+1}(t)\,\mathrm{d}t&>\frac{1}{1-\gamma}\bigg(1-\frac{4\nu(4\nu+1)}{(2\nu-1)(1-\gamma)}\frac{1}{x}\bigg)\mathrm{e}^{-\gamma x}x^\nu I_\nu(x), \quad \nu>\tfrac{1}{2}.
\end{align}
These inequalities are stronger than inequalities (\ref{ineqb2}), (\ref{ineqb3}) and (\ref{ineqb12}), because $I_{\nu+1}(x)<I_\nu(x)$, $x>0$, $\nu>-\frac{1}{2}$ (see (\ref{Imon})). Inequality (\ref{ineqb21}) follows by in part (ii) of the proof of Theorem \ref{tiger1} applying inequality (\ref{Imon}) to bound the third integral in (\ref{55555}), rather than the first integral. Inequality (\ref{ineqb22}) follows from in part (iii) of the proof applying inequality (\ref{Imon}) to the second integral in (\ref{1stint}), rather than the first integral. Examining the proof of inequality (\ref{ineqb12}) below, it can be seen that this modification that allows us to obtain inequality (\ref{ineqb22}) rather than inequality (\ref{ineqb3}) also allows us to obtain inequality (\ref{ineqb23}).

We note that setting $\nu=0$ in inequality (\ref{ineqb21}) yields the neat inequality
\begin{equation*}\int_0^x \mathrm{e}^{-\gamma t}I_1(t)\,\mathrm{d}t>\frac{1}{1-\gamma}(\mathrm{e}^{-\gamma x} I_0(x)-1), \quad x>0,\; 0<\gamma<1.
\end{equation*}
\end{remark}

\begin{remark}Unlike the other bounds presented in this section, inequality (\ref{ineqb1}) is only valid for $x\geq x_*$, where $x_*>\frac{1}{1-\gamma}$.  Crucially, the bound is valid for all $\nu>-\frac{1}{2}$, $0<\gamma<1$, though.  That the bound holds only in the region $x\geq x_*$ is sufficient for our goal of deriving uniform bounds for the expressions (\ref{term1}) and (\ref{term2}) (and a related expression) in Section \ref{sec3}. Interestingly, it is through these uniform bounds derived in Section \ref{sec3} that we obtain inequalities (\ref{ineqb10}) and (\ref{ineqb11}), which hold for all $\nu>-\frac{1}{2}$, $0<\gamma<1$ and all $x>0$.  
\end{remark}

\begin{remark}\label{rem1}The inequalities of Theorems \ref{tiger1} and \ref{tiger2} for the integral (\ref{besint0}) complement those of Theorem 2.1 of \cite{gaunt ineq1} and Theorem 2.3 of \cite{gaunt ineq3}.  We provide a discussion here. Throughout this remark $0<\gamma<1$.

The only other lower bound in the literature is the following one of \cite{gaunt ineq3}: $\int_0^x \mathrm{e}^{-\gamma t}t^{\nu}I_\nu(t)\,\mathrm{d}t>\mathrm{e}^{-\gamma x}x^\nu I_{\nu+1}(x)$, $x>0$, $\nu>-\tfrac{1}{2}$. As $I_\nu(x)>0$ for $x>0$, $\nu\geq-1$, it follows that inequality (\ref{ineqb4}) improves on this inequality. 

Inequality (2.19) of \cite{gaunt ineq3} states that, for $x>0$,
\begin{align}\label{gau9}\int_0^x \mathrm{e}^{-\gamma t}t^{\nu}I_\nu(t)\,\mathrm{d}t&<\frac{\mathrm{e}^{-\gamma x}x^\nu}{(2\nu+1)(1-\gamma)}\Big(2(\nu+1)I_{\nu+1}(x)-I_{\nu+3}(x)\Big), \quad \nu\geq\tfrac{1}{2},\\
\label{gau1}&<\frac{2(\nu+1)}{(2\nu+1)(1-\gamma)}\mathrm{e}^{-\gamma x}x^\nu I_{\nu+1}(x), \quad \nu\geq\tfrac{1}{2},
\end{align}
with the same inequalities being valid for all $\nu>-\frac{1}{2}$ in the case $\gamma=0$ (see inequalities (2.17) and (2.18) of \cite{gaunt ineq3}). Also, combining inequalities (2.3) and (2.5) of \cite{gaunt ineq1} yields the following upper bound: for $x>0$,
\begin{equation}\label{gau2}\int_0^x \mathrm{e}^{-\gamma t}t^{\nu}I_\nu(t)\,\mathrm{d}t<\frac{1}{1-\gamma}\mathrm{e}^{-\gamma x}x^\nu I_\nu(x), \quad \nu\geq\tfrac{1}{2}.
\end{equation}
Inequalities (\ref{ineqb10}) and (\ref{ineqb11}) extend the range of validity of inequalities (\ref{gau1}) and (\ref{gau2}) from $\nu\geq\frac{1}{2}$ to $\nu>-\frac{1}{2}$ at the expense of larger multiplicative constants. Unlike inequalities (\ref{gau9}) and (\ref{gau2}), our inequalities (\ref{ineqb10}) and (\ref{ineqb11}) are not tight in the limit $x\rightarrow\infty$, although they are of the correct asymptotic order $O(x^{\nu-1/2}\mathrm{e}^{(1-\gamma) x})$ as $x\rightarrow\infty$. The multiplicative constant of (\ref{ineqb11}) is half that of (\ref{ineqb10}), although (\ref{ineqb10}) has the advantage of also having the correct asymptotic order $O(x^{2\nu+1})$ as $x\downarrow0$, whereas (\ref{ineqb11}) is $O(x^{2\nu})$ as $x\downarrow0$.

The inequalities derived in this paper together with those presented in this remark allow for a number of two-sided inequalities to be stated for the integral (\ref{besint0}).
A neat example is that, for $x>0$,
\begin{equation}\label{gau3}\mathrm{e}^{-\gamma x}x^\nu\sum_{k=0}^\infty \gamma^k I_{\nu+k+1}(x)<\int_0^x \mathrm{e}^{-\gamma t}t^{\nu}I_\nu(t)\,\mathrm{d}t<\frac{1}{1-\gamma}\mathrm{e}^{-\gamma x}x^\nu I_\nu(x), \quad \nu\geq\tfrac{1}{2}.
\end{equation}
We used \emph{Mathematica} to compute the relative error in approximating the integral $F_{\nu,\gamma}(x)=\int_0^x \mathrm{e}^{-\gamma t}t^{\nu}I_\nu(t)\,\mathrm{d}t$ by the upper bound in (\ref{gau3}), which we denote by $U_{\nu,\gamma}(x)$, and the lower bound in (\ref{gau3}) truncated at the fifth term in the sum,  $L_{\nu,\gamma}(x)=\mathrm{e}^{-\gamma x}x^\nu\sum_{k=0}^4 \gamma^k I_{\nu+k+1}(x)$. The results are reported in Tables \ref{table1} and \ref{table2}. We observe that, for given $x$ and $\nu$, the relative error in approximating $F_{\nu,\gamma}(x)$ by either $L_{\nu,\gamma}(x)$ or $U_{\nu,\gamma}(x)$ increases as $\gamma$ increases. We see that, for given $x$ and $\gamma$, the relative error in approximating $F_{\nu,\gamma}(x)$ by $L_{\nu,\gamma}(x)$ decreases as $\nu$ increases, whilst the relative error in approximating $F_{\nu,\gamma}(x)$ by $U_{\nu,\gamma}(x)$ increases as $\nu$ increases. For given $\nu$ and $\gamma$, the relative error in approximating $F_{\nu,\gamma}(x)$ by $U_{\nu,\gamma}(x)$ decreases as $x$ increases.  This error will approach 0 as $x\rightarrow\infty$, because the bound is tight in this limit.  However, the bound performs poorly for `small' $x$.  Indeed, a simple asymptotic analysis using (\ref{Itend0}) shows that $\frac{U_{\nu,\gamma}(x)}{F_{\nu,\gamma}(x)}\sim\frac{2\nu+1}{(1-\gamma)x}$, as $x\downarrow0$, meaning that the relative error blows up in this limit. The lower bound performs better for `small' $x$, which can be seen because $\lim_{x\downarrow0}\big(1-\frac{L_{\nu,\gamma}(x)}{F_{\nu,\gamma}(x)}\big)=\frac{1}{2(\nu+1)}$.  As a result of truncating the lower bound in (\ref{gau3}) at the fifth term, we lose some accuracy for larger values of $x$, particularly for larger $\gamma$. For example, in the case $\gamma=0.75$, $\sum_{k=0}^\infty 0.75^k=4$ and $\sum_{k=0}^4 0.75^k=3.0508$. Using this and the limiting forms (\ref{eqeq1}) and (\ref{Itendinfinity}) we have that $\lim_{x\rightarrow\infty}\big(1-\frac{L_{\nu,0.75}(x)}{F_{\nu,0.75}(x)}\big)=0.2373$, for all $\nu>-\frac{1}{2}$, whereas the relative error in this limit in approximating $F_{\nu,0.75}(x)$ using the lower bound in (\ref{gau3}) is in fact 0. 
For the cases $\nu=1$ and $\nu=2.5$ we see the relative error decreases down to this limit as $x$ gets larger (after initially increasing for smaller $x$), whilst for the $\nu=5$ and $\nu=10$ cases, the relative error is still increasing from the initial value of $\frac{1}{2(\nu+1)}$ and does not reach the value of 0.2373 for $x\leq100$.

\begin{table}[h]
\centering
\caption{\footnotesize{Relative error in approximating $F_{\nu,\gamma}(x)$ by $L_{\nu,\gamma}(x)$.}}
\label{table1}
{\scriptsize
\begin{tabular}{|c|rrrrrrr|}
\hline
 \backslashbox{$(\nu,\gamma)$}{$x$}      &    0.5 &    5 &    10 &    15 &    25 &    50 & 100   \\
 \hline
$(1,0.25)$ & 0.2563 & 0.2141 & 0.1423 & 0.1028 & 0.0656 & 0.0346 & 0.0182 \\
$(2.5,0.25)$ & 0.1459 & 0.1403 & 0.1100 & 0.0864 & 0.0591 & 0.0329 & 0.0177  \\
$(5,0.25)$ & 0.0846 & 0.0872 & 0.0780 & 0.0670 & 0.0503  & 0.0302 & 0.0169  \\
$(10,0.25)$ & 0.0459 & 0.0481 &  0.0473 & 0.0445 & 0.0378  & 0.0257 & 0.0155 \\   
  \hline
$(1,0.5)$ &0.2644  & 0.2848 & 0.2294 & 0.1846 & 0.1341 & 0.0869 & 0.0602  \\
$(2.5,0.5)$ & 0.1494 & 0.1756 & 0.1625 & 0.1428  & 0.1133 & 0.0791 &  0.0570 \\
$(5,0.5)$ & 0.0860 & 0.1025 & 0.1052 & 0.1005 & 0.0881 & 0.0680 & 0.0522 \\
$(10,0.5)$ & 0.0464 & 0.0533 & 0.0577  & 0.0591 &  0.0580 & 0.0515 & 0.0440 \\ 
\hline
$(1,0.75)$ & 0.2726 & 0.3756 & 0.3829 & 0.3683 & 0.3371 & 0.2953 & 0.2683  \\
$(2.5,0.75)$ & 0.1530 & 0.2211 & 0.2504 & 0.2604 & 0.2640 & 0.2581 & 0.2500 \\
$(5,0.75)$ & 0.0874 & 0.1214 & 0.1470 & 0.1639 & 0.1850 & 0.2084 & 0.2226  \\
$(10,0.75)$ & 0.0468 & 0.0592 & 0.0717  &0.0829  & 0.1028  & 0.1400 & 0.1774 \\  
  \hline
\end{tabular}}
\end{table}

\begin{table}[h]
\centering
\caption{\footnotesize{Relative error in approximating $F_{\nu,\gamma}(x)$ by $U_{\nu,\gamma}(x)$.}}
\label{table2}
{\scriptsize
\begin{tabular}{|c|rrrrrrr|}
\hline
 \backslashbox{$(\nu,\gamma)$}{$x$}      &    0.5 &    5 &    10 &    15 &    25 &    50 & 100   \\
 \hline
$(1,0.25)$ & 6.8497 & 0.2472 & 0.0864 & 0.0520 & 0.0292 & 0.0139 & 0.0068 \\
$(2.5,0.25)$ & 15.7858 & 0.8889 & 0.3548 & 0.2155 & 0.1197 & 0.0565 & 0.0274  \\
$(5,0.25)$ & 28.0748 & 2.0493 & 0.8480 & 0.5129 & 0.2806 & 0.1300 & 0.0625  \\
$(10,0.25)$ & 54.7097 & 4.5473 & 1.9626  & 1.1871 & 0.6377  & 0.2868 & 0.1351 \\   
  \hline
$(1,0.5)$ & 10.4043 & 0.4345 & 0.1459 & 0.0834 & 0.0452 & 0.0212 & 0.0103  \\
$(2.5,0.5)$ & 22.2524 & 1.3875 & 0.5574 & 0.3359 & 0.1842 & 0.0858 & 0.0414   \\
$(5,0.5)$ & 42.1550 & 3.1131 & 1.2980 & 0.7858 & 0.4286 & 0.1972 & 0.0943 \\
$(10,0.5)$ & 82.0875 & 6.8444 & 2.9641  & 1.7968 &  0.9663 & 0.4339 & 0.2037 \\ 
\hline
$(1,0.75)$ & 22.0780 & 1.0751 & 0.3851 & 0.2089 & 0.1019 & 0.0444 & 0.0210   \\
$(2.5,0.75)$ & 44.6563 & 2.9211 & 1.2068 & 0.7284 & 0.3933 & 0.1783 &  0.0845 \\
$(5,0.75)$ & 84.3964 & 6.3211 & 2.6722 & 1.6281 & 0.8891 & 0.4056 & 0.1918   \\
$(10,0.75)$ & 164.2213 & 13.7414 & 5.9790  & 3.6381 &  1.9647 & 0.8827 & 0.4126 \\  
  \hline
\end{tabular}}
\end{table}
\end{remark}

\section{Uniform bounds for some expressions involving the integral}\label{sec3}


In this section, we apply the upper bound (\ref{ineqb1}) to obtain uniform bounds for expressions (\ref{term1}) and (\ref{term2}).  Our upper uniform bounds will lead to technical advances in Stein's method for VG approximation \cite{gaunt vg3}. In addition, we obtain uniform bounds for a related expression; the upper bound will enable us to prove inequality (\ref{ineqb10}). Before doing so, we collect some inequalities for products of modified Bessel functions.  Inequality (\ref{bdsjbc1}) is  given in the proof of Theorem 5 of \cite{gaunt ineq2}, and is a simple consequence of Theorem 4.1 of \cite{hartman}. Inequality (\ref{bdsjbc}) is proved in Lemma 3 of \cite{gaunt ineq2}. Other results and inequalities for the product $I_{\nu}(x)K_{\nu}(x)$ are given in \cite{baricz,baricz2}.

For $x\geq0$,
\begin{equation}\label{bdsjbc1}0\leq xK_{\nu}(x)I_\nu(x)<\frac{1}{2}, \quad \nu>\tfrac{1}{2},
\end{equation}
and
\begin{equation}\label{bdsjbc}\frac{1}{2}< xK_{\nu+1}(x)I_\nu(x)\leq1, \quad \nu\geq-\tfrac{1}{2}.
\end{equation}
We will also need the following lemma.

\begin{lemma}For $x>0$,
\begin{equation}\label{bdsjbc2}\frac{1}{2}< xK_{\nu+2}(x)I_\nu(x)<1+\frac{2\nu+3}{x}, \quad \nu\geq-\tfrac{1}{2}.
\end{equation}
\end{lemma}

\begin{proof}The lower bound follows from the lower bound of (\ref{bdsjbc}) by inequality (\ref{cake}). To prove the upper bound, we note the following inequality of \cite{segura}: for $x>0$,
\begin{equation*}\frac{K_\mu(x)}{K_{\mu-1}(x)}<\frac{\mu-\frac{1}{2}+\sqrt{(\mu-\frac{1}{2})^2+x^2}}{x}<1+\frac{2\mu-1}{x}, \quad \mu>\tfrac{1}{2}.
\end{equation*}
Using this inequality and the upper bound of (\ref{bdsjbc}) we obtain, for $x>0$,
\begin{align*}xK_{\nu+2}(x)I_\nu(x)=\frac{K_{\nu+2}(x)}{K_{\nu+1}(x)}\cdot xK_{\nu+1}(x)I_\nu(x)<1+\frac{2\nu+3}{x},
\end{align*}
as required.
\end{proof}

With inequalities (\ref{bdsjbc1})--(\ref{bdsjbc2}) and the upper bound (\ref{ineqb1}) at hand, we are now in a position to prove the following theorem.

\begin{theorem}\label{notfin} Suppose that $-1<\beta<0$ and $\nu\geq\frac{1}{2}$. Then, for $x\geq0$,
\begin{equation}\label{term000}\frac{\mathrm{e}^{-\beta x}K_{\nu+2}(x)}{x^{\nu-1}}\int_0^x \mathrm{e}^{\beta t}t^{\nu}I_\nu(t)\,\mathrm{d}t<\frac{2(\nu+1)}{(2\nu+1)(1+\beta)}.
\end{equation}
Suppose now that $\nu>-\frac{1}{2}$.  Then
\begin{align}\label{term00}\mathrm{max}\bigg\{\frac{1}{2(1+\beta)},\frac{2(\nu+1)}{2\nu+1}\bigg\}&\leq\sup_{x\geq0}\bigg\{\frac{\mathrm{e}^{-\beta x}K_{\nu+2}(x)}{x^{\nu-1}}\int_0^x \mathrm{e}^{\beta t}t^{\nu}I_\nu(t)\,\mathrm{d}t\bigg\}<\frac{2\nu+7}{(2\nu+1)(1+\beta)}, 
\end{align}
and
\begin{align}
\label{term10}\frac{1}{2(1+\beta)}&\leq\sup_{x\geq0}\bigg\{\frac{\mathrm{e}^{-\beta x}K_{\nu+1}(x)}{x^{\nu-1}}\int_0^x \mathrm{e}^{\beta t}t^{\nu}I_\nu(t)\,\mathrm{d}t\bigg\}<\frac{2\nu+7}{2(2\nu+1)(1+\beta)}, \\
\label{term20}\frac{1}{2(1+\beta)}&\leq\sup_{x\geq0}\bigg\{\frac{\mathrm{e}^{-\beta x}K_{\nu}(x)}{x^{\nu-1}}\int_0^x \mathrm{e}^{\beta t}t^{\nu}I_\nu(t)\,\mathrm{d}t\bigg\}< \frac{2\nu+7}{2(2\nu+1)(1+\beta)}.
\end{align}
\end{theorem}

\begin{proof} (i) By the integral inequality (\ref{gau1}), we have that
\begin{align*}\frac{\mathrm{e}^{-\beta x}K_{\nu+2}(x)}{x^{\nu-1}}\int_0^x \mathrm{e}^{\beta t}t^{\nu}I_\nu(t)\,\mathrm{d}t&<\frac{\mathrm{e}^{-\beta x}K_{\nu+2}(x)}{x^{\nu-1}}\cdot\frac{2(\nu+1)}{(2\nu+1)(1+\beta)}e^{\beta x}x^\nu I_{\nu+1}(x)\\
&=\frac{2(\nu+1)}{(2\nu+1)(1+\beta)}x K_{\nu+2}(x) I_{\nu+1}(x)\leq\frac{2(\nu+1)}{(2\nu+1)(1+\beta)},
\end{align*}
where we used the upper bound in inequality (\ref{bdsjbc}) in the final step.

\vspace{2mm}

\noindent{(ii)} We now prove the lower bounds in (\ref{term00})--(\ref{term20}).  Recall from (\ref{eqeq1}) that, for $-1<\beta<0$ and $\nu>-\frac{1}{2}$,
\begin{equation*}\int_0^x \mathrm{e}^{\beta t}t^{\nu}I_\nu(t)\,\mathrm{d}t\sim\frac{1}{\sqrt{2\pi}(1+\beta)}x^{\nu-1/2}\mathrm{e}^{(1+\beta)x}, \quad x\rightarrow\infty.
\end{equation*}
Combining this limiting form with the limiting form (\ref{Ktendinfinity}) then gives that, as $x\rightarrow\infty$,
\begin{align*}\frac{\mathrm{e}^{-\beta x}K_{\nu+n}(x)}{x^{\nu-1}}\int_0^x \mathrm{e}^{\beta t}t^{\nu}I_\nu(t)\,\mathrm{d}t&\sim \sqrt{\frac{\pi}{2}}x^{1/2-\nu}\mathrm{e}^{-(1+\beta)x}\cdot\frac{1}{\sqrt{2\pi}(1+\beta)}x^{\nu-1/2}\mathrm{e}^{(1+\beta)x}\\
&=\frac{1}{2(1+\beta)},
\end{align*}
where $n\in\mathbb{R}$.  This gives us the lower bounds in (\ref{term10}) and (\ref{term20}), and one of two possible lower bounds in (\ref{term00}). To obtain one of the other possible lower bounds in (\ref{term00}), we examine the behaviour as $x\downarrow0$. Using the limiting forms (\ref{Itend0}) and (\ref{Ktend0}), we have that, as $x\downarrow0$,
\begin{align*}\frac{\mathrm{e}^{-\beta x}K_{\nu+2}(x)}{x^{\nu-1}}\int_0^x \mathrm{e}^{\beta t}t^{\nu}I_\nu(t)\,\mathrm{d}t\sim\frac{ 2^{\nu+1}\Gamma(\nu+2)}{x^{2\nu+1}}\int_0^x\frac{t^{2\nu}}{\Gamma(\nu+1)2^\nu}\,\mathrm{d}t=\frac{2(\nu+1)}{2\nu+1}.
\end{align*} 
On the other hand, the expressions in (\ref{term10}) and (\ref{term20}) can be seen from (\ref{Ktend0}) to be $o(1)$ as $x\downarrow0$.  


\vspace{2mm}

\noindent{(iii)} We now prove the upper bound in (\ref{term00}). We already have an upper bound that is valid for $\nu\geq\frac{1}{2}$ in (\ref{term000}), so we can restrict our attention to the case $-\frac{1}{2}<\nu<\frac{1}{2}$. We obtain our bound by bounding the expression
\[\frac{\mathrm{e}^{-\beta x}K_{\nu+2}(x)}{x^{\nu-1}}\int_0^x \mathrm{e}^{\beta t}t^{\nu}I_\nu(t)\,\mathrm{d}t\]
for $x\in[0,x_*)$ and $x\in[x_*,\infty)$, where $x_*=\frac{2}{1+\beta}$ (note that $x_*>\frac{1}{1+\beta}$). Let us first obtain a bound for $x\in[0,x_*)$. Note that
\begin{equation*}\frac{\partial}{\partial \beta}\bigg(\frac{\mathrm{e}^{-\beta x}K_{\nu+2}(x)}{x^{\nu-1}}\int_0^x \mathrm{e}^{\beta t}t^{\nu}I_\nu(t)\,\mathrm{d}t\bigg)=\frac{\mathrm{e}^{-\beta x}K_{\nu+2}(x)}{x^{\nu-1}}\int_0^x (t-x)\mathrm{e}^{\beta t}t^{\nu}I_\nu(t)\,\mathrm{d}t< 0.
\end{equation*}
As $-1<\beta<0$, we therefore have that, for $0\leq x<x_*$,
\begin{align*}\frac{\mathrm{e}^{-\beta x}K_{\nu+2}(x)}{x^{\nu-1}}\int_0^x \mathrm{e}^{\beta t}t^{\nu}I_\nu(t)\,\mathrm{d}t&< \frac{\mathrm{e}^{ x}K_{\nu+2}(x)}{x^{\nu-1}}\int_0^x \mathrm{e}^{- t}t^{\nu}I_\nu(t)\,\mathrm{d}t\\
&=\frac{1}{2\nu+1}x^2K_{\nu+2}(x)\big(I_\nu(x)+I_{\nu+1}(x)\big)\\
&\leq\frac{2x_*+2\nu+3}{2\nu+1}=\frac{1}{2\nu+1}\bigg(2\nu+3+\frac{4}{1+\beta}\bigg),
\end{align*}
where we used (\ref{intfor}) to evaluate the integral and the upper bounds in inequalities (\ref{bdsjbc}) and (\ref{bdsjbc2}) to obtain the second inequality.

Suppose now that $x\geq x_*$. Let $M_{\nu,-\beta}(x_*)$ be defined as in (\ref{mng}) (with $\gamma=-\beta$). Then, by inequality (\ref{ineqb1}),
\begin{align*}\frac{\mathrm{e}^{-\beta x}K_{\nu+2}(x)}{x^{\nu-1}}\int_0^x \mathrm{e}^{\beta t}t^{\nu}I_\nu(t)\,\mathrm{d}t&<\frac{\mathrm{e}^{-\beta x}K_{\nu+2}(x)}{x^{\nu-1}}\cdot M_{\nu,-\beta}(x_*)\mathrm{e}^{\beta x}x^\nu I_{\nu+1}(x) \\
&=M_{\nu,-\beta}(x_*)xK_{\nu+2}(x)I_{\nu+1}(x)\\
&\leq M_{\nu,-\beta}(x_*)\\
&=\max\bigg\{\frac{2}{2\nu+1}\bigg(\nu+1+\frac{2}{1+\beta}\bigg),\frac{2}{1+\beta}\bigg\}\\
&=\frac{2}{2\nu+1}\bigg(\nu+1+\frac{2}{1+\beta}\bigg),
\end{align*}
where we used inequality (\ref{bdsjbc}) to obtain the second inequality, and that $-\frac{1}{2}<\nu<\frac{1}{2}$ in the last step. Combining our bounds, we have that, for $x\geq0$ and $-\frac{1}{2}<\nu<\frac{1}{2}$,
\begin{align}&\frac{\mathrm{e}^{-\beta x}K_{\nu+2}(x)}{x^{\nu-1}}\int_0^x \mathrm{e}^{\beta t}t^{\nu}I_\nu(t)\,\mathrm{d}t\nonumber\\
&\quad<\mathrm{max}\bigg\{\frac{1}{2\nu+1}\bigg(2\nu+3+\frac{4}{1+\beta}\bigg),\frac{2}{2\nu+1}\bigg(\nu+1+\frac{2}{1+\beta}\bigg)\bigg\}\nonumber\\
\label{pokl}&\quad=\frac{1}{2\nu+1}\bigg(2\nu+3+\frac{4}{1+\beta}\bigg)<\frac{2\nu+7}{(2\nu+1)(1+\beta)}.
\end{align}
Combining inequalities (\ref{term000}) and (\ref{pokl}) (and noting that the upper bound in (\ref{pokl}) is greater than the upper bound in (\ref{term000})) yields inequality (\ref{term00}).

\vspace{2mm}

\noindent{(iv)} The proof of the upper bound in (\ref{term20}) is similar to the proof of the upper bound in (\ref{term10}). Let $x_*=\frac{2}{1+\beta}$. Recall that we already have an upper bound for the case $\nu\geq\frac{1}{2}$ in inequality (\ref{term1}). We therefore also restrict our attention here to the case $-\frac{1}{2}<\nu<\frac{1}{2}$. Arguing similarly to before, we have that, for $0\leq x< x_*$,
\begin{align}\frac{\mathrm{e}^{-\beta x}K_{\nu+1}(x)}{x^{\nu-1}}\int_0^x \mathrm{e}^{\beta t}t^{\nu}I_\nu(t)\,\mathrm{d}t&<\frac{1}{2\nu+1}x^2K_{\nu+1}(x)\big(I_\nu(x)+I_{\nu+1}(x)\big)\nonumber\\
\label{lab1}&<\frac{x_*}{2\nu+1}\bigg(1+\frac{1}{2}\bigg)=\frac{3}{(2\nu+1)(1+\beta)},
\end{align}
where we used inequalities (\ref{bdsjbc1}) and (\ref{bdsjbc}) to obtain the second inequality.
Suppose now that $x\geq x_*$. Then, by inequality (\ref{ineqb1}),
\begin{align}\frac{\mathrm{e}^{-\beta x}K_{\nu+1}(x)}{x^{\nu-1}}\int_0^x \mathrm{e}^{\beta t}t^{\nu}I_\nu(t)\,\mathrm{d}t&<M_{\nu,-\beta}(x_*)xK_{\nu+1}(x)I_{\nu+1}(x)< \frac{1}{2}M_{\nu,-\beta}(x_*)\nonumber\\
\label{lab2}&=\frac{1}{2\nu+1}\bigg(\nu+1+\frac{2}{1+\beta}\bigg)<\frac{\nu+3}{(2\nu+1)(1+\beta)},
\end{align}
where we used inequality (\ref{bdsjbc1}) to obtain the second inequality, and, as in the proof of the upper bound in (\ref{term00}), we used that $-\frac{1}{2}<\nu<\frac{1}{2}$.
Finally, we note that the upper bounds (\ref{term1}), (\ref{lab1}) and (\ref{lab2}) are all bounded above by $\frac{2\nu+7}{2(2\nu+1)(1+\beta)}$ for $\nu>-\frac{1}{2}$, and this gives us our upper bound in (\ref{term10}).

\vspace{2mm}

\noindent{(v)} We obtain the upper bound in (\ref{term20}) as a direct consequence of the upper bound in (\ref{term10}) from an application of inequality (\ref{cake}).
\end{proof}

\begin{remark}Examining the proof of the upper bound in (\ref{term00}), we see that we could improve the bound by choosing $x_*>\frac{1}{1+\beta}$ to be such that
\begin{equation*}\frac{2x_*+2\nu+3}{2\nu+1}=\frac{x_*}{(1+\beta)x_*-1}.
\end{equation*}
This equation reduces to a quadratic equation for $x_*$, for which the solution takes a more complicated form than our choice of $x_*=\frac{2}{1+\beta}$, leading to more a complex upper bound than (\ref{term00}). For given values of $\nu$ and $\gamma$, it would also be possible to optimise the choice of $x_*$ in the derivation of the upper bound in (\ref{term10}). However, our choice of $x_*=\frac{2}{1+\beta}$ has the advantage of allowing us to obtain simple upper bounds that hold for all $\nu>-\frac{1}{2}$ and $-1<\beta<0$.
\end{remark}

We end this section by using some of the upper bounds of Theorem \ref{notfin} to give a short proof of Theorem \ref{tiger2}.

\vspace{3mm}

\noindent{\emph{Proof of Theorem \ref{tiger2}.}} (i) From the upper bound in (\ref{term00}) (with $\beta=-\gamma$) we obtain the following inequality: for $x>0$, $\nu>-\frac{1}{2}$, $0<\gamma<1$,
\begin{align*}\int_0^x\mathrm{e}^{-\gamma t}t^\nu I_\nu(t)\,\mathrm{d}t<\frac{2\nu+7}{(2\nu+1)(1+\beta)}\frac{\mathrm{e}^{-\gamma x}x^{\nu-1}}{K_{\nu+2}(x)},
\end{align*}
and using the inequality $\frac{1}{K_{\nu+2}(x)}<2xI_{\nu+1}(x)$, which is a rearrangement of the lower bound in (\ref{bdsjbc}), yields inequality (\ref{ineqb10}).

\vspace{2mm}

\noindent{(ii) The proof is the same as part (i), but we use the upper bound in (\ref{term10}), rather than the upper bound in (\ref{term00}), and then apply the inequality $\frac{1}{K_{\nu+1}(x)}<2xI_{\nu}(x)$.

\vspace{2mm}

\noindent{(iii) The proof is very similar to the proof of that of inequality (\ref{ineqb3}), with the only difference being that we use inequality (\ref{ineqb10}) to bound the integral on the right-hand side of (\ref{1stint0}), rather than inequality (\ref{gaunt219}).  In addition to giving the alternative bound (\ref{ineqb12}), this extends the range of validity of the bound to $\nu>\frac{1}{2}$. That inequality (\ref{ineqb12}) is tight as $x\rightarrow\infty$ can be proved by the same argument as the one used in part (vi) of the proof of Theorem \ref{tiger1}.
\hfill $\square$

\appendix

\section{Elementary properties of modified Bessel functions}\label{appa}

Here we list standard properties of modified Bessel functions that are used throughout this paper.  All formulas can be found in \cite{olver}, except for the inequalities.

The modified Bessel functions of the first kind $I_\nu(x)$ and second kind $K_\nu(x)$ are defined, for $\nu\in\mathbb{R}$ and $x>0$, by
\begin{equation}\label{Idefn}I_{\nu} (x) =  \sum_{k=0}^{\infty} \frac{(\frac{1}{2}x)^{\nu+2k}}{\Gamma(\nu +k+1) k!},
\end{equation}
and
\begin{equation*} K_\nu(x)=\int_0^\infty \mathrm{e}^{-x\cosh(t)}\cosh(\nu t)\,\mathrm{d}t.
\end{equation*}
The modified Bessel functions $I_{\nu}(x)$ and $K_{\nu}(x)$ are both regular functions of $x\in\mathbb{R}$.  For $x>0$, the functions $I_{\nu}(x)$ and $K_{\nu}(x)$ are positive for $\nu\geq-1$ and all $\nu\in\mathbb{R}$, respectively.  The modified Bessel function of the first kind $I_\nu(x)$ satisfies the following identities and differentiation formulas:
\begin{align}I_{-n}(x)&=I_n(x), \quad n\in\mathbb{Z},\nonumber \\
\label{Iidentity}I_{\nu +1} (x) &= I_{\nu -1} (x) - \frac{2\nu}{x} I_{\nu} (x), \\
\label{diffone}\frac{\mathrm{d}}{\mathrm{d}x} (x^{\nu} I_{\nu} (x) ) &= x^{\nu} I_{\nu -1} (x),\\
\label{diff0}\frac{\mathrm{d}}{\mathrm{d}x} ( I_{0} (x) ) &=  I_{1} (x),
\end{align}
and the integration formula
\begin{equation}\label{intfor}\int_0^x \mathrm{e}^{-t}t^\nu I_\nu(t)\,\mathrm{d}t=\frac{\mathrm{e}^{-x}x^{\nu+1}}{2\nu+1}\big(I_\nu(x)+I_{\nu+1}(x)\big), \quad x>0,\; \nu>-\tfrac{1}{2}.
\end{equation}
The modified Bessel functions have the following asymptotic behaviour:
\begin{align}\label{Itend0}I_{\nu} (x) &\sim \frac{(\frac{1}{2}x)^\nu}{\Gamma(\nu +1)}\bigg(1+\frac{x^2}{4(\nu+1)}\bigg), \quad x \downarrow 0, \: \nu\notin\{-1,-2,-3,\ldots\}, \\
\label{Itendinfinity}I_{\nu} (x) &\sim \frac{\mathrm{e}^{x}}{\sqrt{2\pi x}}, \quad x \rightarrow\infty, \: \nu\in\mathbb{R}, \\
\label{Ktend0}K_{\nu} (x) &\sim \begin{cases} 2^{|\nu| -1} \Gamma (|\nu|) x^{-|\nu|}, & \quad x \downarrow 0, \: \nu \not= 0, \\
-\log x, & \quad x \downarrow 0, \: \nu = 0, \end{cases} \\
\label{Ktendinfinity} K_{\nu} (x) &\sim \sqrt{\frac{\pi}{2x}} \mathrm{e}^{-x}, \quad x \rightarrow \infty, \: \nu\in\mathbb{R}.
\end{align}
Let $x > 0$. Then the following inequalities hold:
\begin{align}\label{Imon}I_{\nu+1} (x) &< I_{\nu} (x), \quad \nu \geq -\tfrac{1}{2}, \\
\label{cake}K_{\nu+1} (x) &> K_{\nu} (x), \quad \nu > -\tfrac{1}{2}.
\end{align}
Inequality (\ref{cake}) is given in \cite{ifant}. Inequality (\ref{Imon}) can be found in \cite{jones} and \cite{nasell}, which extends a result of \cite{soni}.  A survey of related inequalities for modified Bessel functions are given by  \cite{baricz24}, and refinements of inequalities (\ref{Imon}) and (\ref{cake}) are given in \cite{segura} and references therein.

\subsection*{Acknowledgements}
The author is supported by a Dame Kathleen Ollerenshaw Research Fellowship.

\end{document}